\documentclass[11pt,oneside]{article}
\usepackage{yfonts}
\usepackage{amsmath,amssymb,amsthm,units,stmaryrd,color}
\usepackage{graphicx}
\usepackage[all]{xy}
\newtheorem{theorem}{Theorem}[section]

\newtheorem{definition}[theorem]{Definition}
\newtheorem{lemma}[theorem]{Lemma}
\newtheorem{corollary}[theorem]{Corollary}
\newtheorem{proposition}[theorem]{Proposition}

\def\worms{{\mathbb W}}
\def\gllam{\ensuremath{{\sf GLP}_{\Lambda}}}

\def\tila{C}

\def\therest{B}
\def\twoprime{D}
\def\threeprime{E}
\newcommand{\bt}{\begin{theorem}}
\newcommand{\et}{\end{theorem}}
\newcommand{\bl}{\begin{lemma}}
\newcommand{\el}{\end{lemma}}

\newenvironment{enumr}{

\begin{enumerate}     }{\end{enumerate}

}

\newcommand{\cW}{\mathcal{W}}
\newcommand{\cL}{\mathcal{L}}
\newcommand{\ga}{\alpha}
\newcommand{\gb}{\beta}
\newcommand{\gw}{\omega}

\newcommand{\mods}{\mathop{\rm mod}}

\newcommand{\glp}{{\ensuremath{\mathsf{GLP}}}\xspace}

\newcommand{\ord}{{\ensuremath{\sf{On}}}\xspace}

\usepackage{xspace}

\newcommand{\pa}{\ensuremath{{\mathrm{PA}}}\xspace}

\newcommand{\gl}{{\ensuremath{\textup{\textsf{GL}}}}\xspace}
\newcommand{\J}{{\ensuremath{\mathsf{J}}}\xspace}

\newcommand{\ea}{\ensuremath{{\rm{EA}}}\xspace}

\newcommand{\la}{\langle}
\newcommand{\ra}{\rangle}

\newcommand{\NF}{{{\mathbb W}}^\circ}

\def\lb{\left\llbracket}
\def\rb{\right\rrbracket}

\def\fmodels{\xymatrix{
\ar@{|=}[r]^{<\omega}& } }
\def\nmodels{\xymatrix{
\ar@{|=}[r]^{N}& } }
\def\<{\left <}

\def\>{\right >}

\DeclareSymbolFont{AMSb}{U}{msb}{m}{n}
\DeclareMathSymbol{\N}{\mathbin}{AMSb}{"4E}
\DeclareMathSymbol{\Z}{\mathbin}{AMSb}{"5A}
\DeclareMathSymbol{\R}{\mathbin}{AMSb}{"52}
\DeclareMathSymbol{\Q}{\mathbin}{AMSb}{"51}
\DeclareMathSymbol{\I}{\mathbin}{AMSb}{"49}
\DeclareMathSymbol{\C}{\mathbin}{AMSb}{"43}

\begin{document}

\date{\today}
\title{On provability logics with linearly ordered modalities}
\author{Lev D. Beklemishev\footnote{V.A.~Steklov Mathematical Institute, RAS; Moscow M.V. Lomonosov State University; NRU Higher School of Economics; {\tt  bekl@mi.ras.ru }} , David Fern\'{a}ndez-Duque\footnote{Group for Logic, Language and Computation, University of Seville, {\tt dfduque@us.es}}, Joost J. Joosten\footnote{Dept. L\`ogica, Hist\`oria i Filosofia de la Ci\`encia,
Universitat de Barcelona, {\tt jjoosten@ub.edu}}}
\maketitle


\begin{abstract}
We introduce the logics $\gllam$, a generalization of Japaridze's polymodal provability logic $\glp_\omega$ where $\Lambda$ is any linearly ordered set representing a hierarchy of provability operators of increasing strength.

We shall provide a reduction of these logics to
$\glp_{\omega}$ yielding among other things a finitary proof of the normal
form theorem for the variable-free fragment of $\gllam$ and
the decidability of $\gllam$ for recursive orderings
$\Lambda$. Further, we give a restricted
axiomatization of the variable-free fragment of $\gllam$.
\end{abstract}

\section{Introduction}

The provability logic $\glp_\Lambda$ with transfinitely many modalities $\la \ga\ra$, for all ordinals $\ga<\Lambda$,
generalizes the well-known provability logic \glp denoted $\glp_\gw$ in this paper \cite{Japaridze:1986,Boolos:1993:LogicOfProvability}. 
The logic $\glp_\omega$ has been used to carry out a proof-theoretic analysis of Peano Arithmetic and related theories using the approach of provability algebras initiated in \cite{Beklemishev:2004}. A natural next class of theories to analyze with this new approach are predicative theories such as the second order theories of iterated arithmetical comprehension and $\textsf{ATR}_0$. The first necessary step towards analyzing predicative theories with provability algebras was made in \cite{Beklemishev:2005:VeblenInGLP} where logics $\glp_\Lambda$, for an arbitrary ordinal $\Lambda$, were introduced and it was shown that the variable-free fragments of these logics yield a natural ordinal notation system up to the ordinal $\Gamma_0$.

Assuming an ordinal $\Lambda$ to be represented, ordinals of a possibly larger class can be denoted by modal formulas (called \emph{words} or \emph{worms}) of the form
$$\la\ga_1\ra\la\ga_2\ra\dots\la\ga_n\ra\top,$$ where $\ga_i<\Lambda$, identified modulo provable equivalence in $\glp_\Lambda$. The
ordering between two words $A$ and $B$ is naturally defined by
$$A<_0 B\iff \glp_\Lambda\vdash B\to \la 0\ra A.$$ It was shown that
this ordering is a well-ordering, and basic formulas for the
computation of the order types of its initial segments in terms of
Veblen ordinal functions were found in \cite{Beklemishev:2005:VeblenInGLP}.

Since then, the logics $\glp_\Lambda$ and their ordinal notation systems have been studied in much more detail (see
\cite{FernandezJoosten:2012:TuringProgressions,FernandezJoosten:2012:ModelsOfGLP,FernandezJoosten:2012:Hyperations}).
Most importantly, suitable Kripke models for the variable-free fragment of $\glp_\Lambda$ generalizing the so-called Ignatiev model for
$\glp_\gw$ \cite{Ignatiev:1993:StrongProvabilityPredicates} have been
developed. Also, the completeness of $\glp_\Lambda$ w.r.t.\
topological semantics has been proved
\cite{BeklemishevGabelaia:2011:TopologicalCompletenessGLP,Fernandez:2012:TopologicalCompleteness}.
Some of these papers used the normal form results from
\cite{Beklemishev:2005:VeblenInGLP}.

Sections \ref{section:AlternativeAxiomatization} and \ref{section:NormalFormClosedFormulas} of the present paper is in many respects a `recasting' of the part of
\cite{Beklemishev:2005:VeblenInGLP} devoted to the normal forms for
the variable-free fragment of $\glp_\Lambda$ and to its
axiomatizations. The main reason to have such a recasting is that the
exposition in \cite{Beklemishev:2005:VeblenInGLP} was at some places
overly sketchy, to the extent that some parts of the arguments were
only hinted at. The main such omission was the proof of the fact
that the ordering $<_0$ on words was irreflexive, or
equivalently the fact that any individual word was consistent with
$\glp_\Lambda$. Modulo this claim, the rest of the arguments in the
paper were purely syntactical or dealt with ordinal computations. For this consistency result one would naturally use some kind of semantics, which were not available at the time for $\Lambda>\omega$ (but see \cite{FernandezJoosten:2012:ModelsOfGLP}).

Another reason for having a recast of parts of \cite{Beklemishev:2005:VeblenInGLP} is that the authors of \cite{FernandezJoosten:2012:ModelsOfGLP} needed certain results --in particular, Corollaries \ref{theorem:zeroDiamondYieldWorms} and \ref{theorem:CompletenessFollowsFromSoundness} of the current paper -- that follow from the line of reasoning presented in \cite{Beklemishev:2005:VeblenInGLP}. 
However, a proof of these corollaries could not be given without revisiting and sharpening various results from \cite{Beklemishev:2005:VeblenInGLP}.

Moreover, it was remarked in \cite{Beklemishev:2005:VeblenInGLP} that the irreflexivity of $<_0$ follows, for example, from
any arithmetically sound interpretation of $\glp_\Lambda$ w.r.t.\ a sequence of strong provability predicates. Indeed, the
existence of such interpretations was obvious at least for constructive ordinals $\Lambda$. On the other hand, a proof appealing to such an
interpretation is necessarily based on the assumption of soundness of a fairly strong extension of Peano Arithmetic and thus cannot be formalized in Peano Arithmetic itself. For proof-theoretic applications we would like to have an ordering representation whose elementary properties such as  irreflexivity are provable by finitary means (e.g., in Primitive Recursive Arithmetic). Alternative proofs based on the use of Ignatiev-like models or topological models for $\glp_\Lambda$ suffer from the same drawback.

In this paper we remedy this situation and provide a different purely modal finitary proof of irreflexivity based on a reduction of $\glp_\Lambda$ to $\glp_\gw$, for which such a finitary proof is known \cite{BeklemishevJoostenVervoort:2005:FinitaryTreatmentGLP}. We also prove the conservativity of $\glp_\Lambda$ over any of its restrictions to a subset of modalities. This reduction uses the methods of \cite{Beklemishev:2010}.

The exposition of the normal form theorem for variable-free formulas in $\glp_\Lambda$ in this paper is also slightly different from the one in \cite{Beklemishev:2004,Beklemishev:2005:VeblenInGLP}. Namely, the normal forms are defined in a `positive' way, which helps, in particular, to eliminate the assumption of irreflexivity at some places where it is not necessary. Finally, we provide a more restricted axiomatization of the variable-free fragment of $\glp_\Lambda$ than the one in \cite{Beklemishev:2005:VeblenInGLP}.

An additional novelty of this paper is that the results can be stated and proved in a more general context of logics with linearly ordered sets of modalities. Thus, from the outset we introduce and work with a generalization of $\gllam$ to the case when $\Lambda$ is an arbitrary, not necessarily well-founded, linear ordering. So far, proof-theoretic interpretations of such logics have not been investigated; however it seems likely that they can appear, for example, in the study of progressions of theories defined along recursive linear orderings without infinite hyperarithmetical descending sequences (see, e.g., \cite{FefSpe62}).

\section{The logic $\gllam$ and its fragments}

In this section we shall introduce the formal systems that we will
study throughout the paper. Our logics depend on a
parameter, usually denoted $\Lambda$, which is a linear order of the form $\langle|\Lambda|,<\rangle$. They then contain a modality
$[\alpha]$ for each $\alpha\in |\Lambda|$. In analogy to the set-theoretic treatment of ordinals, we will identify $\Lambda$ with an upper bound for its elements and often write $\alpha < \Lambda$ instead of $\alpha \in |\Lambda|$; elements of $|\Lambda|$ will sometimes be called {\em modals}. Note, however, that unlike previous studies of $\gllam$, we allow for $\Lambda$ to be an arbitrary linear order.

We will also introduce some important fragments of $\gllam$. These fragments are easier to work with from a technical point of view,
yet they already contain much of the crucial information about the
full logic, as we shall see.

\subsection{The logics $\gllam$}

The full language ${\cL}_\Lambda$ is built from propositional
variables in a countably infinite set $\mathbb P$ and the constant
$\top$ together with the Boolean connectives $\neg,\wedge$ and a
unary modal operator $[\alpha]$ for each $\alpha\in\Lambda$. As is
customary, other Boolean operators may be defined in the standard
way and we write $\langle\alpha\rangle$ as a shorthand for
$\neg[\alpha]\neg$.

We will use $\mods \phi$ to denote the set of elements of $|\Lambda|$
appearing in $\phi$ and $\max\phi$ to be the
maximum of these modals. We also use $l(\phi)$ to denote the {\em
length} of $\phi$, defined in a standard way, and $w(\phi)$ to be
its {\em width,} that is, the number of modals appearing in
$\phi$.

\begin{definition}[$\gllam$]
Given a linear order $\Lambda=\langle|\Lambda|,<\rangle$, $\gllam$ is the logic over ${\cL}_\Lambda$ given by the
following rules and axioms:

\begin{itemize}
\item All substitution instances of propositional tautologies,

\item For all $\alpha,\beta\in|\Lambda|$ and formulas
$\chi,\psi\in{\cL}_\Lambda$,
\[
\begin{array}{lll}
$(\romannumeral1)$ & [\alpha] (\chi \to \psi) \to ([\alpha] \chi \to  [\alpha]\psi) & \\
$(\romannumeral2)$ & {}[ \alpha ] ([\alpha] \chi \to \chi) \to [\alpha] \chi &\\
$(\romannumeral3)$ & {}[\alpha] \chi \to [\beta] [\alpha] \chi & \mbox{for  $\alpha \leq \beta$}\\
$(\romannumeral4)$ & \langle \alpha \rangle \chi \to [\beta] \langle \alpha \rangle \chi & \mbox{for  $\alpha < \beta$,}\\
$(\romannumeral5)$ & {}[\alpha] \chi \to [\beta] \chi &\mbox{for
$\alpha \leq \beta$}.
\end{array}
\]
\item Modus Ponens and the necessitation rule $\displaystyle\frac{\chi}{[\alpha] \chi}$ for each modality $\alpha \in |\Lambda|$.
\end{itemize}
\end{definition}
This definition contains certain redundancies: Axiom
$(\romannumeral3)$ is clearly derivable in presence of the others,
and necessitation for $0$ would suffice given Axiom
$(\romannumeral5)$. However, it will be convenient to state these
principles separately.


\subsection{Kripke semantics}

Kripke models give us a transparent and convenient interpretation
for many modal logics. A {\em Kripke frame} is a structure
$\mathfrak F=\<W,\<R_\lambda\>_{\lambda<\Lambda}\>$, where $W$ is a
set and $\<R_\lambda\>_{\lambda<\Lambda}$ a family of binary
relations on $W$. A {\em valuation} on $\mathfrak F$ is a function
$\lb\cdot\rb:{\cL}_\Lambda\to \mathcal P(W)$ such that
\[
\begin{array}{lcllcl}
\lb\bot\rb&=&\varnothing\ &\qquad
\lb\neg\phi\rb&=&W\setminus\lb\phi\rb\\
\lb\phi\wedge\psi\rb&=&{}\lb\phi\rb\cap\lb\psi\rb\ &\qquad
\lb\<\lambda\>\phi\rb&=&R^{-1}_\lambda\lb\phi\rb.
\end{array}
\]

A {\em Kripke model} is a Kripke frame equipped with a valuation
$\lb\cdot\rb$. Note that propositional variables may be assigned
arbitrary subsets of $W$. Often we will write $\<\mathfrak
F,\lb \cdot \rb\>,x\Vdash\psi$ instead of $x\in\lb \psi\rb$ or even just $x\Vdash\psi$ if the context allows us to. As usual, $\phi$ is
{\em satisfied} on $\<\mathfrak
F,\lb \cdot \rb\>$ if $\lb\phi\rb\not=\varnothing$,
and {\em valid} on $\<\mathfrak
F,\lb \cdot \rb\>$ if $\lb\phi\rb=W$. The latter case shall be denoted by $\<\mathfrak F,\lb \cdot \rb\> \models \psi$.

We shall also use the notion of frame validity in that $\mathfrak
F , x \models \psi$ denotes that $\<\mathfrak
F,\lb \cdot \rb\>,x\Vdash\psi$ for any valuation $\lb \cdot \rb$. Likewise, $\mathfrak
F  \models \psi$ denotes that $\mathfrak
F , x \models \psi$ for all $x$ in $\mathfrak
F$.

$\gllam$ has no non-trivial Kripke models, but its \emph{variable-free} or {\em closed
fragment} (defined below) does
\cite{Ignatiev:1993:StrongProvabilityPredicates}. We will use a
sublogic $\J_\Lambda$ of $\gllam$ that is sound and complete
w.r.t.\ a suitable class of finite frames called \emph{J-frames}.
The logics $\J_\Lambda$ can be obtained from the given axiomatization of $\gllam$
by replacing the monotonicity axiom schema (v) by the following
schema (derivable in $\gllam$):
$$[\alpha] \phi \to [\ga] [\gb] \phi, \text{for  $\alpha \leq
\beta$}.$$ This system has been introduced in
\cite{Beklemishev:2010} just for the language $\cL_\gw$. Although it
is easy to see that the Kripke model completeness theorem for $\J_\gw$
proved in \cite{Beklemishev:2010} holds more generally, we will
actually use it only for the logic $\J_\gw$.

A Kripke frame is called a \emph{$J_\Lambda$-frame} if, for all
$\gb<\ga<\Lambda$,
\begin{itemize}
\item $R_\ga$ is a conversely well-founded, transitive ordering relation on
$W$;
\item $\forall x,y\:( x R_\ga y \Rightarrow \forall z\:(xR_\gb z \Leftrightarrow
yR_\gb z))$;
\item $\forall x,y\:( xR_\ga y \ \mathop{\&}\  yR_\gb z \Rightarrow xR_\ga
z)$.
\end{itemize}
A $\J_\Lambda$-frame is called \emph{finite} if so is the set of its nodes $W$. A \emph{$\J_\Lambda$-model} is a Kripke model based on a $\J_\Lambda$-frame.

The following is proved in \cite{Beklemishev:2010} for $\Lambda=\gw$, but holds more generally with the same proof.
\begin{proposition}\
\begin{enumerate}
\item If $\J_\Lambda\vdash\phi$ then $\phi$ is valid in all $J_\Lambda$-models;
\item
If $\J_\Lambda\nvdash\phi$ then $\phi$ is not valid in some finite $J_\Lambda$-model.
\end{enumerate}
\end{proposition}

\subsection{Fragments of $\gllam$}\label{section:ClosedFragments}

There are two particular families of sublogics of $\gllam$ which we
will focus on later. The first is the fragment without variables,
which as we shall see is already quite expressive:

\begin{definition}[Closed fragment]
We denote by $\cL^0_\Lambda$ the sublanguage of $\cL_\Lambda$ whose
formulas do not contain propositional variables (only $\top$).

$\gllam^0$ denotes the intersection of $\gllam$ with
$\cL^0_\Lambda$.
\end{definition}

That is, $\glp^0_\Lambda$ is the set of provable formulas of
$\gllam$ that do not contain any propositional variables. It
is clear that any closed formula $\psi$ provable in
$\gllam^0$ can also be proved using proofs and axioms
without variables. For, given a proof $\pi$ of $\psi$, we can
substitute $\top$ (or $\bot$) for the propositional variables that
occur in $\pi$. After substitution we still have a proof of $\psi$.

The second fragment is the restriction to a subset of all
modals, which is especially useful when this subset is finite.

\begin{definition}
For any subset $S\subseteq|\Lambda|$, let $\cL_S$ denote the language
with the set of modalities $\{[\xi]:\xi\in S\}$, and let $\glp_S$ be
the logic\footnote{In principle we should include $\Lambda$ as a second parameter since these logics depend on the specific ordering, but we will let this be given by context.} given by the restriction of the axioms and rules of
$\gllam$ to $\cL_S$.
\end{definition}

As we shall see in Section \ref{section:ReductionForFullGLP}, any
provable formula of $\cL_S$ is also provable within $\glp_S$.
However, this is not as immediate as in the case of the closed
fragment.


\section{Reduction of $\gllam$ to its finite fragments}
\label{section:ReductionForFullGLP}

Here we show that $\gllam$ is conservative over any of its
fragments obtained by restricting the language to a subset of its
modalities.

Clearly, if $S$ is a set of ordinals, $\glp_S$ is only notationally
different from $\glp_\gb$ where $\gb$ is the order type of $S$. More
precisely, let $\xi_\ga$ be the $\ga$-th element of $S$ and let
$\xi(\phi)$ denote the result of replacing in a formula $\phi$ (in
the language $\cL_\gb$) each modality $[\ga]$ by $[\xi_\ga]$.
Similarly, let $\xi^{-1}(\psi)$ denote the inverse operation. Then
the following lemma is obvious.

\bl \label{isom}\
\begin{enumr}
\item $\glp_\gb\vdash\phi$ iff $\glp_S\vdash\xi(\phi)$;
\item
$\glp_S\vdash \psi$ iff $\glp_\gb\vdash\xi^{-1}(\psi)$.
\end{enumr}
\el

The conservation result is now stated as follows.

\bt \label{conserv} Given a linear order $\Lambda$, $S\subseteq|\Lambda|$ and a formula $\phi$ in $\cL_S$,
$\gllam\vdash\phi$ iff $\glp_S\vdash\phi$. \et

\begin{proof}
A proof will proceed in two steps. First, we prove the
conservativity of $\glp_\gw$ over any of its finite fragments.
Secondly, we will use a purely syntactic argument to lift this
result to arbitrary fragments of $\gllam$.

We are going to use the following standard reduction of $\glp_\gw$
to $\J_\gw$ (see \cite{Beklemishev:2010}). Let $\phi$ be a
$\glp_\gw$-formula, and let $\{[m_i]\phi_i:i<s\}$ be all the boxed
subformulas of $\phi$ with $m_i \leq m_j$ whenever $i<j$. Denote:
\[
M^+(\phi):=M(\phi)\land \bigwedge_{i\leq m_s} [i]M(\phi),
\]
where
\[
M(\phi):=\bigwedge_{i< s} \bigwedge_{m_i<j\leq m_s} ([m_i]\phi_i\to [j]\phi_i).
\]

The following result is proved in \cite{Beklemishev:2010} using
Kripke model techniques. Alternative proofs (using the topological
and the arithmetical semantics, respectively) can be found in
\cite{Beklemishev:2011:SimplifiedArithmeticalCompleteness,BeklemishevGabelaia:2011:TopologicalCompletenessGLP}.
The proof in \cite{Beklemishev:2010} has the advantage of being
formalizable in Elementary Arithmetic.\footnote{The formula
$M(\phi)$ is misspelled in \cite{Beklemishev:2010}.}

\bl \label{one} $\glp_\gw\vdash \phi\iff \J_\gw\vdash
M^+(\phi)\to\phi.$ \el

We are going to show here that the formula $M^+$ can be replaced by
a formally weaker one: $N^+(\phi):=N(\phi)\land \bigwedge_{i<s}
[m_i]N(\phi)$ and
$$N(\phi):=\bigwedge_{i<s} \bigwedge_{i<j<s} ([m_i]\phi_i\to [m_j]\phi_i).$$
Notice that $N^+(\phi)$ is in the language of $\phi$.

\bl \label{two} $\glp_\gw\vdash \phi\iff \J_\gw\vdash
N^+(\phi)\to\phi.$ \el

\begin{proof} Suppose $\J_\gw\nvdash N^+(\phi)\to\phi$. Then there is a
finite $\J_\gw$-model $\cW$ with a node $r$ such that $\cW,r\Vdash
N^+(\phi)$ and $\cW,r\nVdash\phi$. Replace each relation $R_k$ in
$\cW$ by $\emptyset$, for all $k\notin S:=\{m_0,\dots,m_{s-1}\}$.
The result is still a $\J_\gw$-model (denoted $\cW'$), and the
forcing of formulas in the language of $\phi$ is everywhere the
same.


Finally, we observe that $M^+(\phi)$ is true at $r$. It is
sufficient to show that each implication $[m_i]\phi_i\to [j]\phi_i$,
for $m_i<j\leq m_s$, holds at each point $x\in \cW'$ reachable from
$r$. We observe that such an $x$ is either $r$ itself or is
reachable by one of the relations $R_{m_i}$, for $i<s$. Since
$r\Vdash N(\phi)\land [m_i]N(\phi)$ we have $x\Vdash N(\phi)$.
Hence, if $j\in S$ we have $x\Vdash [m_i]\phi_i\to [j]\phi_i$ as
required. However, if $j\notin S$ the relation $R_j$ is empty, and
thus $x\Vdash [j]\phi_i$ trivially. Thus, Lemma \ref{two} follows
from Lemma \ref{one}. \end{proof}

For any $S\subseteq \gw$ let $\J_S$ denote the restriction of the
logic $\J_\gw$ to the language $\cL_S$.

\bl  \label{three} For any formula $\phi$ in $\cL_S$,
$\J_\gw\vdash\phi$ iff $\J_S\vdash\phi$. \el

\begin{proof}
Only the (only if) part needs to be proved. Assume
$\J_S\nvdash\phi$. Consider any $\J_\gw$-model $\cW$ in the
restricted language $\cL_S$ such that $\cW\not \models\phi$. For each
$i\notin S$, define a new relation $R_i$ on $\cW$ by letting
$R_i=\emptyset$. The expanded model $\cW'$ is a model of $\J_\gw$
and $\cW'\not \models \phi$. Hence, $\J_\gw\nvdash\phi$.
\end{proof}

From Lemmas \ref{two} and \ref{three} we obtain the conservativity
of $\glp_\gw$ over its fragments.

\begin{corollary} \label{cons-omega}
Let $S\subseteq \gw$ and $\phi$ be a formula in $\cL_S$. Then
$\glp_\gw\vdash\phi$ iff $\glp_S\vdash\phi$.
\end{corollary}

Now we turn to the general case and prove Theorem \ref{conserv}.
Assume $\phi$ is in $\cL_S$ and $\gllam\vdash\phi$. Let
$R\subseteq|\Lambda|$ be the set of all modals occurring in the
given derivation of $\phi$. The same derivation shows that
$\glp_R\vdash\phi$. Since $R$ is finite, we can assume it is
enumerated by some function $\xi:\{0,\dots,n-1\}\to R$. Let
$\psi:=\xi^{-1}(\phi)$. By Lemma \ref{isom} we obtain $\glp_n\vdash
\psi$ and hence $\glp_\gw\vdash \psi$.

Let $F$ be the set of modals occurring in $\phi$. Obviously,
$F\subseteq R$ and $G:=\xi^{-1}(F)\subseteq \gw$. Therefore, by
Corollary \ref{cons-omega} $\glp_{G}\vdash\psi$. It follows that
$\glp_{\xi G}\vdash \xi(\psi)$, that is, $\glp_F\vdash \phi$. Since
$F\subseteq S$ we conclude that $\glp_S\vdash \phi$, as required.
This completes the proof of Theorem \ref{conserv}.
\end{proof}

For any formula $\phi$ let $\hat{\phi}$ denote $\xi^{-1}(\phi)$,
where $\xi:\{0,\dots,n-1\}\to F$ enumerates the set $F$ of all
modals occurring in $\phi$. Applying Theorem \ref{conserv} to $F$
we obtain the following corollary.

\begin{corollary} For any $\phi$, $\gllam\vdash\phi$ iff
$\glp_n\vdash\hat{\phi}$ iff $\glp_\omega\vdash\hat{\phi}$. \end{corollary}

\begin{proof} By Theorem \ref{conserv}, $\gllam\vdash\phi$ iff
$\glp_F\vdash\phi$, whereas by Lemma \ref{isom} the latter is
equivalent to $\glp_n\vdash\hat{\phi}$.
\end{proof}

By this corollary, the logic $\gllam$ inherits many nice
properties proved for $\glp_\gw$. Let us state a few explicitly. Below, the corollaries follow directly from their counterparts as proven for $\glp_\gw$ \cite{Ignatiev:1993:StrongProvabilityPredicates,Beklemishev:2010:OnCraigInterpolation,Shamkanov:2011:InterpolationProperties}.

\begin{corollary}
$\gllam$ is a decidable logic, provided $\Lambda$ has a recursive presentation.
\end{corollary}

\begin{corollary}
$\gllam$ enjoys Craig interpolation: If $\psi(\vec p, \vec q)$ and $\phi (\vec q, \vec r)$ are $\cL_{\Lambda}$-formulas with all variables among the distinct variables $\vec p,\vec q, \vec r$ with
$\gllam \vdash \psi(\vec p, \vec q) \to \phi (\vec q, \vec r)$, then there is some formula $\theta( \vec {q} )$ whose variables are all among $\vec q$ such that
\[
\gllam \vdash \Big( \psi(\vec p, \vec q) \to \theta( \vec {q} )\Big) \ \ \wedge \ \ \Big( \theta( \vec {q} ) \to \phi (\vec q, \vec r) \Big).
\]
\end{corollary}

\begin{corollary}
$\gllam$ has unique fixpoints: Let $\psi(\vec p, q)$ be a formula of $\cL_\Lambda$ where $q$ only occurs under the scope of a modality.  Then, there exists some $\phi (\vec p)$ such that $\psi(\vec p, q/ \phi (\vec p))$ is $\gllam$-provably equivalent to $\phi (\vec p)$. Moreover, this is provable within $\gllam$ itself:
\[
\gllam \vdash \boxdot (q \leftrightarrow \phi (\vec p)) \ \leftrightarrow \ \boxdot (q \leftrightarrow \psi(\vec p, q)).
\]
\end{corollary}
The standard variations of this theorem like unique solutions to simultaneous fixpoints equations also carry directly through to $\gllam$.

\begin{corollary}
$\gllam$ satisfies the uniform interpolation property: for any $\cL_\Lambda$ -formula $\psi (\vec q, \vec r)$ with distinguished variables $\vec q$ there exists a \emph{uniform interpolant}, that is, a formula $\phi (\vec q)$ such that for any $\theta (\vec q)$ we have
\[
\gllam \vdash \psi (\vec q, \vec r) \to \theta (\vec q) \ \ \ \ \Longleftrightarrow \ \ \ \ \gllam \vdash \phi (\vec q) \to \theta (\vec q).
\]
\end{corollary}



\section{Worms and their normal forms}\label{section:AlternativeAxiomatization}

In this section we study {\em worms,} or iterated consistency
satements, which in a sense form the backbone of the logic $\gllam^0$
(recall that $\gllam^0$ is the fragment of $\gllam$ which contains no
propositional variables). Worms directly code the ordinals needed
for a proof-theoretic analysis of formal theories. Moreover, as we
shall see, every closed formula of $\gllam^0$ can be written as a
Boolean combination of worms.

Many of the results presented here appeared originally in
\cite{Beklemishev:2005:VeblenInGLP}. The main difference is that we
employ a different --but equivalent, as we shall see-- definition of
normal forms on worms. We also include more details than in
\cite{Beklemishev:2005:VeblenInGLP} and do not use the irreflexivity
of the $<_\alpha$ relations.

\begin{definition}[Worms]
The set of words, or worms, is a subset of $\cL_\Lambda^0$ denoted by $\worms$
and is inductively defined as $\top \in \worms$, and $A\in \worms \
\Rightarrow \ \langle\alpha\rangle A\in \worms$ where $\alpha$ is a modal.

We write $\alpha \in A$ to indicate that $\alpha$ occurs somewhere
in the word $A$. By $\worms_{\alpha}$ we denote $\{ A \in \worms \mid \beta
\in A \Rightarrow \beta \geq \alpha \}$.
\end{definition}

It is customary to identify a worm $A$ with the sequence of the
modals in $A$. Thus, $\langle 0 \rangle \langle 2 \rangle \top$
will be associated with just $02$ but we shall also employ any
hybrid form like $\la 0 \ra 2$, etc. We will associate $\top$ with
the empty sequence/word $\epsilon$. Worms owe their name to the heroic worm-battle, a variant of the Hydra
battle (see \cite{Beklemishev:2006}), but they may also be called
\emph{words}.

\subsection{Natural orderings on $\worms_{\alpha}$}

On the set of worms one can define natural order relations.

\begin{definition}
For $A, B \in \worms$ we define $A <_{\alpha} B \ :\Leftrightarrow \gllam
\vdash B \to \langle \alpha \rangle A$.
\end{definition}

It is clear by Axiom $(\romannumeral3)$ that $<_{\alpha}$ is
transitive for each $\alpha$ and by Axiom $(\romannumeral5)$, that
$<_{\beta} \subseteq <_{\alpha}$ for $\alpha \leq \beta$. In
\cite{Beklemishev:2005:VeblenInGLP} it is shown that assuming
irreflexivity for $<_{\alpha}$, the orderings $<_{\alpha}$ define a
well-order order on $\worms_{\alpha}$ modulo provable equivalence, provided $\Lambda$ is itself well-ordered. Thus in this case,
given irreflexivity, the elements of $\worms_\ga$ can be associated with
ordinals.

The next lemma is the basis of a large portion of our reasoning and
we shall use it in the remainder of this paper without explicit
mention.

\begin{lemma}\label{theorem:basicLemma}
\begin{enumerate}\

\item\label{lemma:basicLemma1}
For closed formulas $\phi$ and $\psi$, if $\beta < \alpha$, then \\
$\gllam \vdash (\langle \alpha \rangle \phi \wedge \langle \beta
\rangle \psi) \leftrightarrow \langle \alpha \rangle(\phi \wedge
\langle \beta \rangle \psi)$;

\item\label{lemma:basicLemma2}
For closed formulas $\phi$ and $\psi$, if $\beta < \alpha$, then \\
$\gllam \vdash (\langle \alpha \rangle \varphi \wedge [ \beta] \psi)
\leftrightarrow \langle \alpha \rangle(\varphi \wedge [ \beta ]
\psi)$;

\item\label{lemma:basicLemma0}
$\gllam \vdash AB \to A$

\item\label{lemma:basicLemma3}
If $A\in \worms_{\alpha+1}$, then $\gllam \vdash A \wedge \langle \alpha
\rangle B \leftrightarrow A\alpha B$;

\item \label{lemma:basicLemma4}
If $A, B \in \worms_{\alpha}$ and $\gllam \vdash A \leftrightarrow B$,
then\\ $\gllam \vdash A\alpha C \leftrightarrow B \alpha C$.
\end{enumerate}
\end{lemma}

\begin{proof}
For \ref{lemma:basicLemma1}, we observe that by Axiom
$(\romannumeral4)$ we have $\langle \beta \rangle \psi \to [\alpha]
\langle \beta \rangle \psi$, whence $\langle \alpha \rangle \phi
\wedge \langle \beta \rangle \psi \to \langle \alpha \rangle (\phi
\wedge \langle \beta \rangle \psi)$. For the other direction, we
note that $\langle \alpha \rangle (\phi \wedge \langle \beta \rangle
\psi) \to \langle \alpha \rangle \langle \beta \rangle \psi$ and the
antecedent implies $\langle \beta \rangle \psi$ by Axiom
$(\romannumeral3)$.

The proof of \ref{lemma:basicLemma2} is similar. By Axiom
$(\romannumeral3)$ we see that $[\beta]\psi \to
[\alpha][\beta]\psi$, whence $\langle \alpha \rangle \phi \wedge
[\beta] \psi \to \langle \alpha \rangle(\phi \wedge [\beta]\psi)$.
For the other direction, we use Axiom $(\romannumeral4)$ to get
$\langle \beta \rangle \neg \psi \to [\alpha ]\langle \beta \rangle
\neg \psi$. Thus,
\[
\begin{array}{lll}
\langle \alpha \rangle(\phi \wedge [\beta]\psi) \wedge \langle \beta \rangle \neg \psi& \to & \langle \alpha \rangle \bot\\
& \to & \bot,
\end{array}
\]
whence $\langle \alpha \rangle(\phi \wedge [\beta]\psi) \to [\beta]
\psi$.

Item \ref{lemma:basicLemma0} is proven by induction on the length of
$A$. For zero length we see that $A = \top$. For the inductive case
we reason in $\gllam$ and consider $\langle \alpha \rangle A B$. By a
necessitation on the induction hypothesis we get $[\alpha] (AB \to
A)$. Using Axiom $(\romannumeral1)$, we see that $\langle \alpha
\rangle A B \wedge [\alpha] (AB \to A) \to \langle \alpha \rangle
A$. We shall later see that in general $\nvdash AB \to B$.

Item \ref{lemma:basicLemma3} follows from repeatedly applying
\ref{lemma:basicLemma1} (from outside in), and Item
\ref{lemma:basicLemma4} follows from Item \ref{lemma:basicLemma3}.
\end{proof}
Using the $<_{\alpha}$ relation we can define a normal form for
worms.

\begin{definition}[worm normal form]
A worm $A \in \worms$ is in WNF (worm normal form) iff
\begin{enumerate}
\item
$A = \epsilon$, or

\item
$A$ is of the form $A_k\alpha \ldots \alpha A_1$ with $\alpha =
\min(A)$, $k\geq1$ and $A_i \in \worms_{\alpha+1}$ such that each $A_i$
is in WNF and moreover
$A_{i+1} \leq_{\alpha+1}A_{i}$ for each $i<k$.
\end{enumerate}
\end{definition}

We note that the definition of WNF refers to provability in
$\gllam$ every time it states $A_{i+1} \leq_{\alpha+1}A_{i}$ : recall
that the latter is short for $\gllam  \vdash A_{i} \to \langle \alpha
+1 \rangle A_{i+1}$ or $\gllam\vdash A_i \leftrightarrow A_{i+1}$. In virtue of Theorem
\ref{conserv} we can replace the use of $\gllam$ by its relevant
fragment of finite signature.

\begin{lemma}\label{theorem:WidthOneIsWNF}
Each worm of width one is in WNF.
\end{lemma}

\begin{proof}
This is immediate if we conceive $\alpha^n$ as
$\epsilon\alpha\epsilon\ldots\epsilon\alpha\epsilon$.
\end{proof}

We emphasize that WNFs on worms are rather similar in form to Cantor
normal forms (CNF) with base $\omega$ on ordinals.  A notable
difference is that where ordinals in CNF have their largest terms on
the left-hand side, worms have their largest ``term" on the
right-hand side.

Lemma \ref{theorem:WormsAreLikeCNFs} below tells us that, in order
to compare two worms in WNF it suffices to compare, just as with
CNFs, the largest non-equal components. As a slight abuse of
notation, we will often write a worm $A$ in the form $A_k \alpha
\ldots A_1$, with the understanding that $A=\epsilon$ when $k=0$ and
$A=A_1$ when $k=1$.

\begin{lemma}\label{theorem:WormsAreLikeCNFs}
Let $A = A_k \alpha \ldots A_1\alpha A'$ be in WNF with $\alpha =
\min(A)$, and each $A_i\in \worms_{\alpha+1}$. Moreover, let $B$ be in
WNF. We have that
\[
\mbox{if } A' <_{\alpha +1} B, \mbox{ then }   A <_{\alpha +1} B.
\]
\end{lemma}

\begin{proof}
By induction on $k$. We write $A$ as $A_{k}\alpha \tila$. As
$A_{k}\alpha \tila$ is in WNF and $\alpha = \min(A)$, we see that
necessarily $\tila$ is of the form $\twoprime\alpha \threeprime$
with $\twoprime \in \worms_{\alpha+1}$ and the $\alpha \threeprime$ part
possibly empty. By the IH (or by assumption in case $k=0$), we see
that $ \tila<_{\alpha+1} B $, from which we obtain
\[
\begin{array}{rlll}
B & \to & \langle \alpha +1 \rangle \tila& \\
 & \to & \langle \alpha +1 \rangle (\twoprime \alpha \threeprime) & \\
 & \to & \langle \alpha +1 \rangle \twoprime \wedge \langle \alpha +1 \rangle \tila& \mbox{as $\twoprime \geq_{\alpha+1} A_{k}$}\\
 & \to & \langle \alpha +1 \rangle A_{k} \wedge \langle \alpha +1 \rangle \tila& \\
 & \to & \langle \alpha +1 \rangle A_{k} \wedge \langle \alpha  \rangle \tila& \\
 & \to & \langle \alpha +1 \rangle A_{k}  \alpha  \tila.& \\
 \end{array}
 \]
In other words, $ A_{k}  \alpha  \tila<_{\alpha +1}B$ and we are
done. Note that the proof also works for $A'=\epsilon$ in which case
$A$ is just of the form $\alpha^m$ for some $m\in \omega$.
\end{proof}

Let us introduce some special notation for worms in WNF.

\begin{definition}
We denote $\worms_{\alpha} \ \cap $ WNF by $\NF_\alpha$.
\end{definition}

\begin{lemma}\label{theorem:LinearOrder}
For all $A, B\in \NF_\alpha$, either $A=B$, $A<_{\alpha}B$ or $B<_{\alpha}A.$
\end{lemma}

\begin{proof}
We may assume that $\alpha \in AB$. For if this were not the case,
we prove the lemma for $A,B \in \worms_{\beta}$ where $\beta=\min(A,B)$
and see that $\gllam \vdash A\to \langle \beta \rangle B$ implies
$\gllam \vdash A\to \langle \alpha \rangle B$. In case $\beta$ does
not exist we have $AB=A=B=\epsilon$.

We will prove the lemma by induction on $w(AB)$. Recall that by our convention, $A_k
\alpha \ldots A_1\alpha A'$ should be understood to denote $A'$ for
$k=0$ and $A_1\alpha A'$ for $k=1$.

For $w(AB)\leq 1$ and $A\neq B$ we see that $ l(A) < l(B)\
\Rightarrow \  A <_{\alpha} B$ thus obtaining our result as either
$l(A) < l(B)$ or $l(B) < l(A)$.

We now consider $w(AB)>1$. Suppose that $A\neq B$. We may assume
that none of $A$ or $B$ is a proper extension of the other, for if,
for example, $B$ were a proper extension of $A$, then $A<_{\alpha}B$
by Axiom $(\romannumeral3)$. Thus, we write $A = A_k\alpha \ldots
A_n \alpha \ldots A_1$ and $B = B_m \alpha \ldots B_n\alpha \ldots
B_1$ where $n$ is the smallest number such that $A_n \neq B_n$. By
the IH we may, w.l.o.g.\ assume that $A_n <_{\alpha+1} B_n$.
Reasoning in \gllam we see that
\[
\begin{array}{lll}
B_n \to \langle \alpha +1 \rangle A_n & \Rightarrow & B_n \wedge \alpha \ldots B_1 \to \langle \alpha +1 \rangle A_n \wedge \alpha \ldots A_1\\
 & \Rightarrow & B_n  \alpha \ldots B_1 \to \langle \alpha +1 \rangle (A_n  \alpha \ldots A_1).\\
\end{array}
\]

By Lemma \ref{theorem:WormsAreLikeCNFs} we conclude that $A
<_{\alpha+1} B_n \alpha \ldots B_1$. As clearly $B_n  \alpha \ldots
B_1\leq_{\alpha} B$ we obtain $A<_{\alpha}B$ as desired.
\end{proof}
Note that it is necessary to require that $A,B \in \worms_{\alpha}$ in
the above lemma: as we shall see, the normal forms $1$ and $01$ are
$<_1$-incomparable. It is easy to see that the proof of the lemma
automatically yields the following corollary.

\begin{corollary}
Consider two worms $A = A_m\alpha \ldots \alpha A_1$ and $B=
B_n\alpha \ldots \alpha B_1$ both in $\NF_\alpha$ with $A_i, B_j \in
\NF_{\alpha+1}$, and not all the $A_i$ nor all the $B_j$ empty. Let
$<^L_{\alpha+1}$ denote the lexicographical ordering on finite
strings over $\worms_{\alpha +1}$ induced by $<_{\alpha +1}$. We have
that
\[
A<_\alpha B \ \Leftrightarrow \ (A_1, \ldots, A_m) <_{\alpha +1}^L
(B_1, \ldots , B_n).
\]
\end{corollary}

The above considerations are sufficient to give an effective procedure for deciding the ordering on worms, provided we have a procedure for ordering $\Lambda$.

\begin{definition}
We call a procedure $\Lambda$-effective if it is effective using an
oracle for deciding $\alpha< \beta$ for $\alpha, \beta \in |\Lambda|$.
\end{definition}

\begin{corollary}\label{theorem:effectiveComparisonOfWorms}
There is a $\Lambda$-effective procedure that compares two worms in
$\NF_{\alpha}$.
\end{corollary}

\begin{proof}
The $\Lambda$-effective decision procedure is already present in the
proof. For $w(AB) \leq 1$ deciding whether $A<_{\alpha} B$ amounts
to counting and comparing the number of symbols in $A$ and $B$. For
$w(AB) > 1$ this amounts to checking for first checking equality.
This we can do, as we can pose oracle queries on elements in the
$\langle |\Lambda|,<\rangle$ ordering. If $A \neq B$, we look at the
first (from the right) non-equal term in $A$ and $B$ and recursively
call upon our decision procedure. Note that in this case $l(AB)$
will diminish so we have an effective bound on the amount of calls
on the decision procedure.
\end{proof}

Next we formulate an obvious corollary to lemma
\ref{theorem:LinearOrder} that will be very useful later on.

\begin{corollary}\label{theorem:ComparingWorms}
For each $A,B \in \NF_{\alpha}$, either $\gllam \vdash \alpha A \to
\alpha B$, or $\gllam \vdash \alpha B \to \alpha A$.
\end{corollary}

\begin{proof}
All implications in this proof refer to implications inside $\gllam$.
By Lemma \ref{theorem:LinearOrder} we have $A\leq_{\alpha}B$ or
$B<_{\alpha}A$. If $A=B$ the implication is clearly provable. If
$A<_{\alpha} B$, then $B\to \alpha A$ whence $\alpha B \to \alpha
\alpha A$, and $\alpha B \to \alpha A$. Likewise, $B<_{\alpha}A$
implies $\alpha B \to \alpha A$.
\end{proof}

\begin{corollary}\label{theorem:WormsClosedUnderConjunction}
Given worms $A,B \in \NF_{\alpha}$, there is a worm $C\in
\worms_{\alpha}$ with $\gllam \vdash A\wedge B \leftrightarrow C$.
Moreover, we have that ${\mods}(C) \subseteq
{\mods}(AB)$, and $l(C)\leq l(AB)$.
\end{corollary}

\begin{proof}
By induction on $w(AB)$. The base case is trivial. For the inductive
case, we assume w.l.o.g.\ that $\alpha \in AB$ and write $A =
A_1\alpha A_2$ and $B = B_1 \alpha B_2$ with at most one of $\alpha
A_2, \alpha B_2$ empty and $A_1, B_1 \in \NF_{\alpha+1}$. We reason
in \gllam. By the IH, we find some $C_1 \leftrightarrow A_1 \wedge
B_1$. By Corollary \ref{theorem:ComparingWorms} we may assume that
$\alpha A_2 \to \alpha B_2$. Thus, we conclude the proof by
\[
\begin{array}{lll}
A \wedge B & \leftrightarrow & A_1\alpha A_2 \wedge B_1\alpha B_2\\
& \leftrightarrow & A_1\wedge \alpha A_2 \wedge B_1 \wedge \alpha B_2\\
& \leftrightarrow & A_1\wedge B_1 \wedge \alpha A_2 \wedge  \alpha B_2\\
& \leftrightarrow & C_1 \wedge \alpha A_2 \wedge  \alpha B_2\\
& \leftrightarrow & C_1 \wedge \alpha A_2\\
& \leftrightarrow & C_1 \alpha A_2\\
\end{array}
\]
\end{proof}

\begin{corollary}\label{theorem:effectiveWormConjunction}
There is a $\Lambda$-effective procedure which, given two worms $A$
and $B$ in WNF, computes a worm $C$ so that $\gllam \vdash
A\wedge B \leftrightarrow C$ with ${\mods}(C) \subseteq
{\mods}(AB)$, and $l(C)\leq l(AB)$.
\end{corollary}

\begin{proof}
The proof of Lemma \ref{theorem:WormsClosedUnderConjunction}
contains a decision procedure. For $w(AB)\leq 1$ computing the
conjunction just amounts to taking the longer of $A$ or $B$.

For $w(AB)>1$ we compute $C$ as dictated by the proof of Lemma
\ref{theorem:WormsClosedUnderConjunction} where we use Corollary
\ref{theorem:effectiveComparisonOfWorms} to decide which of $\alpha
A_2 \to \alpha B_2$ or $\alpha B_2 \to \alpha A_2$ is the case.
\end{proof}

In Lemma \ref{theorem:LinearOrder} we have proved that $<_{\alpha}$
defines a linear order on the set of normal forms of $\worms_{\alpha}$.
We shall next see through a series of lemmata that each worm $A$ is
equivalent in $\gllam$ to one in WNF. Thus, we can drop the condition of
worms being in WNF in various lemmata above
(\ref{theorem:LinearOrder}, \ref{theorem:ComparingWorms}, and
\ref{theorem:WormsClosedUnderConjunction}).

\begin{lemma}\label{theorem:RemoveMultipleMinimals}
For non-empty $A \in \worms_{\alpha+1}$ we have for any $B\in \worms$ that
\[
\gllam \vdash A \alpha B \leftrightarrow A \alpha^nB \ \ \ \mbox{for
$n\in \omega\setminus \{0\}$}.
\]
\end{lemma}

\begin{proof}
By an easy induction on $n$.
\end{proof}

\begin{lemma}\label{theorem:gettingRidOfSmallTerms}
Let $A : = A_{1}\alpha A_0 \therest$ with ($\therest = \epsilon$ or
$\therest = \alpha A'$) and each of $A_{1}, A_0$ in $\worms_{\alpha+1}$.
\[
\mbox{If } \ \ \gllam \vdash A_{1} \to \langle \alpha +1 \rangle A_0,
\mbox{ then } \gllam \vdash A \leftrightarrow A_{1} \therest.
\]
\end{lemma}

\begin{proof}
We assume $\gllam \vdash A_{1} \to \langle \alpha +1 \rangle A_0$.
(The first direction actually holds without the assumption.)

From $A_1\alpha A_0 \therest$ we get $A_1 \wedge A_0\therest$. If
$\therest$ is of the form $\alpha A'$, from $\alpha A_0{\therest}$
we get $\therest$ by repeatedly applying Axiom $(\romannumeral3)$
from inside out. When $\therest = \epsilon$, we have $\therest$
straight away of course. Thus,
\[
\begin{array}{llll}
\gllam \vdash & A&\to&A_1\wedge \alpha A_0\therest\\
 &  & \to & A_1 \wedge \therest\\
 &  & \to & A_1\therest.\\
\end{array}
\]
For the other direction we reason in $\gllam$ and use our assumption
that $A_1 \to \langle \alpha+1 \rangle A_0$.
\[
\begin{array}{lll}
A_1\therest &\to & A_1\therest \wedge \langle \alpha+1 \rangle A_0\\
 &\to & A_1 \wedge \therest \wedge \langle \alpha+1 \rangle A_0\\
  &\to & A_1 \wedge \langle \alpha+1 \rangle A_0 \therest\\
    &\to & A_1 \wedge \langle \alpha \rangle A_0 \therest\\
        &\to & A_1  \alpha A_0 \therest.\\
\end{array}
\]
\end{proof}

\begin{lemma}\label{theorem:ExistNormalForm}
Each worm $A \in \worms$ is equivalent in $\gllam$ to some ${\sf NF}(A)$ in
WNF. Moreover, ${\mods}({\sf NF}(A)) \subseteq
{\mods}(A)$.
\end{lemma}

\begin{proof}
By induction on $l(A)$ we shall prove that each $A \in \worms_{\alpha}$
is equivalent to some $\tila \in \worms_{\alpha}$ with $l(\tila)\leq
l(A)$ and ${\mods}(\tila) \subseteq {\mods}(A)$. For
$l(A)= 0$ we see that $A = \epsilon \in$ WNF. We proceed to prove
the case when $l(A)>0$. All modal reasoning takes place in \gllam.

For $\alpha = \min{(A)}$, we use Lemma
\ref{theorem:RemoveMultipleMinimals} to write $A$ as $A_{k+1}\alpha
A_k\alpha \ldots  A_0$ with $k\geq 0$ and each $A_i \in
\worms_{\alpha+1}$. Recall that $ A_0\alpha \ldots A_0$ just means $
A_0$. By the IH we find some $A_i'\in \worms_{\alpha+1}$ such that
$A'_l\alpha \ldots A'_0$ is in WNF and equivalent to $A_k\alpha
\ldots A_0$. Moreover, we have that $l(A'_l\alpha \ldots A'_0) \leq
l(A_k\alpha \ldots A_0)$. It is easy to see that we also have that
$\alpha A'_l\alpha \ldots A'_0$ is equivalent to $\alpha A_k\alpha
\ldots A_0$.

Again, by the IH, we can find some $\twoprime \in \NF_{\alpha+1}$
which is equivalent to $A_{k+1}$ and with $l(\twoprime) \leq
l(A_{k+1})$.
Clearly we have that
\[
\begin{array}{lll}
\twoprime\alpha A_l' \alpha \ldots A_0' & \leftrightarrow & \twoprime \wedge \alpha A_l' \alpha \ldots A_0'\\
& \leftrightarrow & A_{k+1} \wedge \alpha A_l' \alpha \ldots A_0'\\
& \leftrightarrow & A_{k+1} \wedge \alpha A_k \alpha \ldots A_0\\
& \leftrightarrow & A.\\
\end{array}
\]
If $A_l' \geq_{\alpha+1} \twoprime$ then $\twoprime\alpha A_l'
\alpha \ldots A_0'$ is in WNF. Moreover, $l(\twoprime\alpha A_l'
\alpha \ldots A_0') \leq l(\twoprime) + l(\alpha A_l'\alpha \ldots
A_0')  \leq l(A_{k+1}) + l(\alpha A_k\alpha \ldots A_0)\leq l(A)$
and $\twoprime\alpha A_l' \alpha \ldots A_0' \in \worms_{\alpha}$.

If $A_l' \not \geq_{\alpha+1} \twoprime$ we conclude by Lemma
\ref{theorem:LinearOrder} that $\twoprime >_{\alpha+1}A_l'$. Now we
can apply Lemma \ref{theorem:gettingRidOfSmallTerms} to see that
\[
\begin{array}{lll}
A & \leftrightarrow & \twoprime\alpha A_l' \alpha \ldots A_0'\\
& \leftrightarrow & \twoprime\alpha A_l' \therest\\
& \leftrightarrow & \twoprime\therest.\\
\end{array}
\]
We conclude by yet another call upon the IH to find a WNF in
$\worms_{\alpha}$ equivalent to $\twoprime\therest$ and of length at most
$l(\twoprime\therest)$.

Thus, tranforming a worm into an equivalent one in WNF boils down to
repeatedly shortening the original worm by applying lemmata
\ref{theorem:RemoveMultipleMinimals} and
\ref{theorem:gettingRidOfSmallTerms} whence it is clear that
${\mods}({\sf NF}(A)) \subseteq {\mods}(A)$.
\end{proof}

\begin{corollary}\label{theorem:WNFisEffectivelyComputable}
Given some $\gllam$ worm $A \in \worms_{\alpha}$, there is a
$\Lambda$-computable procedure to obtain a worm $A' \in \NF_\alpha$
with ${\mods}(A') \subseteq {\mods}(A)$ and
$\gllam \vdash A \leftrightarrow A'$.
\end{corollary}

\begin{proof}
We see that the proof of Lemma \ref{theorem:ExistNormalForm}
actually contains a description of this decision procedure. In the
inductive step, whether or not we have to apply Lemma
\ref{theorem:gettingRidOfSmallTerms} can be $\Lambda$-decided in
virtue of Corollary \ref{theorem:effectiveComparisonOfWorms}.
\end{proof}

Now that we have seen that we can $\Lambda$-effectively compute a
WNF, we conclude from Corollary
\ref{theorem:effectiveWormConjunction} that we can $\Lambda$-compute
the conjunction of any two worms $A$ and $B$. In other words, we can
omit the restriction that $A$ and $B$ be in WNF in Corollary
\ref{theorem:effectiveWormConjunction}.


\section{A normal form theorem for closed formulas}\label{section:NormalFormClosedFormulas}

So far in this paper, no irreflexivity of the relations $<_{\alpha}$
has been used in our reasoning. In this section we shall prove that
each closed formula is actually equivalent in $\gllam$ to a Boolean
combination of worms and some important corollaries thereof. In the
proofs, irreflexivity plays an essential role.

\subsection{Irreflexivity}

By {\em irreflexivity} we mean the claim that for no $A\in \worms$ and
for no $\alpha \in |\Lambda|$ do we have $\gllam \nvdash A \to \langle
\alpha \rangle A$. In view of the following result, this is equivalent to demanding that worms be consistent.

\begin{lemma}\label{theorem:LobOnNegatedWorms}
If $\gllam \vdash A \to \langle \alpha \rangle A$, then $\gllam \vdash
\neg A$.
\end{lemma}
\begin{proof}
If we assume $\gllam \vdash A \to \langle \alpha \rangle A$, then we
would get by contraposition and necessitation that $\gllam \vdash
[\alpha ] ([\alpha] \neg A \to \neg A)$. One application of L\"ob's
axiom would yield $\vdash [\alpha] \neg A$. Using the contraposition
of our assumption again, we obtain $\gllam \vdash \neg A$.
\end{proof}

Fortunately, irreflexivity does hold. This is known for well-ordered $\Lambda$, in which case there are many arguments
in the literature as to why that is, each with its advantages and
disadvantages.

\paragraph{Arithmetic interpretations.} In case of $\glp_{\omega}$ all formulas $\psi$ come with a clearly
defined arithmetical interpretation $\psi^{\star}$ where each $[n]$
is interpreted as a natural formalization of ``provable in \ea
together all true $\Pi_n$-sentences'' \cite{Ignatiev:1993:StrongProvabilityPredicates}. The
soundness for this interpretation tells us that for any formula
$\varphi$ and any interpretation $\star$ mapping propositional
variables to sentences in the language of arithmetic we have that
$\glp_\omega \vdash \varphi \Rightarrow \pa \vdash \varphi^{\star}$. In
particular we get for worms $A$ that $\glp_\omega \vdash \neg A \Rightarrow
\pa \vdash \neg A^{\star}$. Now $\neg A^{\star}$ is just an
iteration of inconsistency assertions all of which are not provable
by \pa as everything provable by \pa is actually true. This
reasoning, although using quite some heavy machinery as reflection
over \pa, establishes the irreflexivity of $<_{n}$ in
$\glp_{\omega}$. Recent work by the authors and Dashkov suggests that this may be generalized to larger recursive ordinals than $\omega$, however arithmetic interpretations for non-recursive ordinals or for linear orders that are not well-founded are not currently known.

\paragraph{Kripke semantics.} Kripke semantics for $\glp_{\omega}$ have been studied extensively \cite{Ignatiev:1993:StrongProvabilityPredicates, Joosten:2004:InterpretabilityFormalized, BeklemishevJoostenVervoort:2005:FinitaryTreatmentGLP, Icard:2008:MastersThesis}. Using these
semantics it is easy to see that for each $n\in \omega$, and each
worm $A \in \glp_{\omega}$ we can find a model $\mathcal{M}$ and a
world $x$ of $\mathcal{M}$ where both $A$ and $[n] \neg A$ hold, thus establishing the irreflexivity of $<_n$ in $\glp_{\omega}$.
More recently this has been extended to $\gllam$ for an arbitrary ordinal
$\Lambda$ \cite{FernandezJoosten:2012:KripkeSemanticsGLP0}. One
drawback is that the methods used are not strictly finitary, whereas
\cite{BeklemishevJoostenVervoort:2005:FinitaryTreatmentGLP} gives a
full finitary treatment of $\glp^0_\omega$. Thus the irreflexivity
of $\glp_\omega$ can be proven on strictly finitary grounds. As before, the assumption that $\Lambda$ is well-ordered plays an important role and it is not obvious how one could generalize these methods, however they do have the advantage of working for arbitrary ordinals, including uncountable ones.

\paragraph{Topological semantics.} The same reasoning can also be performed using topological semantics of $\glp_\omega$
\cite{Icard:2009:TopologyGLP, Icard:2008:MastersThesis, BeklemishevGabelaia:2011:TopologicalCompletenessGLP}, which likewise have been generalized to arbitrary ordinals in \cite{FernandezJoosten:2012:ModelsOfGLP}. As before, however, the methods used in the transfinite setting are not strictly finitary and have been developed only for well-ordered $\Lambda$.\\\\

Now that we have provided a reduction from $\gllam$ to
$\glp_\omega$ in Theorem \ref{conserv}, we in particular have a
reduction from $\gllam^0$ to $\glp_\omega^0$. This gives us a
new proof of irreflexivity for the general logic. The present argument is both the first finitary proof
of irreflexivity for infinite orders different from $\omega$, provided that $\Lambda$ (and hence $\gllam$)
can be represented in a finitary framework such as Primitive
Recursive Arithmetic, as well as the first proof of irreflexivity which does not require that $\Lambda$ be well-founded.

\begin{theorem}\label{theorem:Irreflexive}
For each linear order $\Lambda$ and each $\alpha\in|\Lambda|$, the relation $<_{\alpha}$ is irreflexive on $\worms$.
\end{theorem}

\proof The relation $<_n$ is known to be irreflexive over
$\glp_\omega$, and this fact may be proven by finitary means
\cite{Joosten:2004:InterpretabilityFormalized, BeklemishevJoostenVervoort:2005:FinitaryTreatmentGLP}.
Moreover, if for some worm $A$ we had that $\gllam\vdash
(A\to\langle\alpha\rangle A)=\psi$, then we would have that
$\glp_\omega\vdash \hat\psi$, contradicting the irreflexivity of
$<_n$ for some $n$.
\endproof

Thus, we have shown that $<_{\alpha}$ is transitive and irreflexive
and defines a linear order on the worm normal forms in $\worms_{\alpha}$.
In fact, in \cite{Beklemishev:2005:VeblenInGLP} it has been shown to
be a well-order on $\worms_{\alpha}$, if it is irreflexive and $\Lambda$ is well-founded. In particular, if we allow $\Lambda$ to be the clas of all ordinals,
there is a one-one correspondence between normal forms in $\worms$ and
ordinals in \ord. In \cite{FernandezJoosten:2012:WellOrders2} the
relation $<_{\alpha}$ is also studied and seen to be a non-tree-like
partial well-order on $\worms$.

Without using irreflexivity we proved two major results on worms and
WNFs. First, that WNFs are linearly ordered by $<_0$, and second,
that each worm is equivalent to one in WNF. Using irreflexivity we
readily see that the WNFs actually form a strict linear order under
$<_0$ and that each formula is equivalent to a \emph{unique} WNF.

\begin{lemma}\label{theorem:WNFsAreUnique}
Each worm $A$ is equivalent in $\gllam$ to a unique worm ${\sf NF}(A)$
in WNF.
\end{lemma}

\begin{proof}
Suppose for a contradiction that $A$ had over \gllam two different
WNFs ${B}$ and $\tila$. Then, by Lemma \ref{theorem:LinearOrder} and
reasoning in \gllam we may assume that ${B} \to \langle \alpha \rangle
\tila$ where $\alpha = \min(A)$. Thus,
\[
\begin{array}{lll}
A & \to & {B} \\
& \to& \langle \alpha \rangle \tila\\
& \to& \langle \alpha \rangle A,\\
\end{array}
\]
which contradicts irreflexivity.
\end{proof}

Using irreflexivity it also immediate that our new definition of
normal forms is equivalent to the one previously used in the
literature. In the remainder of this paper we shall freely use
irreflexivity.

\subsection{Closed formulas and worms}

In this section we shall show that each closed formula is equivalent
to a Boolean combination of worms. We follow Section 3 of
\cite{Beklemishev:2005:VeblenInGLP} very closely, formulating
slightly stronger versions of the lemmata in
\cite{Beklemishev:2005:VeblenInGLP} leading up to important further
observations.

The first lemma of this section in a sense tells us that whatever
piece of genuine information we add to a worm, this will always
increase the consistency strength of it (equivalently, increase the
corresponding order-type).

\begin{lemma}\label{theorem:AddingWormsIncreasesOrder}
Let $A, A_1,\hdots A_I \in \worms_{\alpha}$ be such that for each $i\leq
I$, $\gllam \nvdash A \to A_i$. Then it follows that $\gllam \vdash A
\wedge \bigvee_{i=1}^I A_i \to \langle \alpha \rangle A$.
\end{lemma}

\begin{proof}
All modal reasoning will be in $\gllam$. By Corollary
\ref{theorem:WormsClosedUnderConjunction} for each $i$, let ${\sf
Conj}(A,A_i)$ be the worm in $\NF_{\alpha}$ that is equivalent to $A
\wedge A_i$. By Lemma \ref{theorem:LinearOrder} we can
$<_{\alpha}$-compare ${\sf Conj}(A,A_i)$ to $A$. However, ${\sf
Conj}(A,A_i) = A$ contradicts $\nvdash A \to A_i$. Likewise, ${\sf
Conj}(A,A_i) <_{\alpha} A$ contradicts the irreflexivity of
$<_{\alpha}$. We conclude that ${\sf Conj}(A,A_i) \to \langle \alpha
\rangle A$ whence $A \wedge A_i \to \langle \alpha \rangle A$. As
$i$ was arbitrary, we obtain $A \wedge \bigvee_{i=1}^I A_i \to
\langle \alpha \rangle A$.
\end{proof}

A direct and nice corollary to this lemma is that worms satisfy a
certain form of disjunction property.

\begin{corollary}\label{theorem:DisjunctionPropertyWorms}
For $A,A_i \in \worms$ we have that
\[
\gllam \vdash A \to \bigvee_{i=1}^I A_i \ \ \ \ \Leftrightarrow \ \ \
\ \mbox{for some $i\leq I$, }\gllam \vdash A \to  A_i.
\]
\end{corollary}

\begin{proof}
We reason about derivability in $\gllam$ by contraposition and suppose
that for each $i\leq I$, $\nvdash A \to A_i$. Then, by Lemma
\ref{theorem:AddingWormsIncreasesOrder} we obtain that $\vdash A
\wedge \bigvee_{i=1}^I A_i\to \langle 0 \rangle A$. Irreflexivity of
$<_0$ imposes that $\nvdash A \to \bigvee_{i=1}^I A_i$, as required.
\end{proof}

\begin{lemma}\label{theorem:WormsClosedUnderDiamonds}
For $A, A_1, \ldots A_k \in \worms_{\alpha}$ we have in $\gllam$ that
either
\begin{itemize}
\item
$\langle \alpha \rangle (A \wedge \bigwedge_i \neg A_i)
\leftrightarrow \langle \alpha \rangle A$, or that
\item
$A \wedge \bigwedge_i \neg A_i \leftrightarrow \bot$ whence also $\langle \alpha
\rangle (A \wedge \bigwedge_i \neg A_i) \leftrightarrow \bot$.
\end{itemize}
\end{lemma}

\begin{proof}
All modal reasoning will concern $\gllam$. In case that for some $i$
we have that $\vdash A\to A_i$, clearly $\langle
\alpha \rangle (A \wedge \bigwedge_i \neg A_i) \leftrightarrow
\bot$. In case that for no $i$, $\vdash A\to A_i$ we apply Lemma
\ref{theorem:AddingWormsIncreasesOrder}:
\[
\begin{array}{llll}
[\alpha ] (A \to \bigvee_i A_i) & \to & [\alpha ] (A \to (A \wedge \bigvee_i A_i)) & \mbox{by Lemma \ref{theorem:AddingWormsIncreasesOrder}}\\
 & \to & [\alpha] (A \to \langle \alpha \rangle A) & \mbox{by L\"ob's axiom}\\
 & \to & [\alpha] \neg A & \\
 & \to & [\alpha ] (A \to \bigvee_i A_i)  \\
\end{array}
\]
Thus, $[\alpha ] (A \to \bigvee_i A_i) \leftrightarrow [\alpha] \neg
A$, whence $\langle \alpha \rangle (A \wedge \bigwedge_i \neg A_i)
\leftrightarrow \langle \alpha \rangle A$.
\end{proof}

\begin{corollary}\label{theorem:generalizedWormsClosedUnderDiamonds}
For any worm $A\in \worms$, and $A_1, \ldots A_k \in \worms_{\alpha}$ we have
in $\gllam$ that either
\begin{itemize}
\item
$\langle \alpha \rangle (A \wedge \bigwedge_i
\neg A_i) \leftrightarrow \langle \alpha \rangle A$, or that
\item
$A \wedge \bigwedge_i \neg A_i \leftrightarrow \bot$ whence also  $\langle \alpha \rangle (A \wedge \bigwedge_i \neg A_i)
\leftrightarrow \bot$.
\end{itemize}
\end{corollary}

\begin{proof}
We can split $A$ into the largest prefix $A_{\alpha}$ of $A$ that
belongs to $\worms_\alpha$ and the remainder $A_{<\alpha}$of $A$.
Consequently, $A_{<\alpha}$ starts with a symbol smaller than
$\alpha$ or is empty and we have $A = A_{\alpha}A_{<\alpha}
\leftrightarrow A_{\alpha} \wedge A_{<\alpha}$. Thus,
\[
\begin{array}{llll}
\langle \alpha \rangle (A \wedge \bigwedge_i \neg A_i) & \leftrightarrow & \langle \alpha \rangle (A_{\alpha} \wedge A_{<\alpha} \wedge \bigwedge_i \neg A_i)& \\
& \leftrightarrow &  A_{<\alpha}\wedge\langle \alpha \rangle (A_{\alpha}  \wedge \bigwedge_i \neg A_i)& \mbox{first case of Lemma \ref{theorem:WormsClosedUnderDiamonds}}\\
& \leftrightarrow &  A_{<\alpha}\wedge\langle \alpha \rangle A_{\alpha}&\\
& \leftrightarrow &  \langle \alpha \rangle (A_{\alpha}\wedge A_{<\alpha})&\\
& \leftrightarrow &  \langle \alpha \rangle A.&\\
\end{array}
\]
Note that in the second case of Lemma
\ref{theorem:WormsClosedUnderDiamonds} we end up with $\bot$ as
desired.
\end{proof}

\begin{lemma}\label{theorem:closedFormulasClosedUnderDiamonds}
Let $\phi(A_1, \ldots, A_n)$ be a Boolean combination of the  worms $A_1,
\ldots, A_n$. Then $\langle \alpha \rangle \phi(A_1, \ldots, A_n)$
is equivalent in $\gllam$ to some formula ${\sf
Diamond}_{\alpha}(\phi)$ which is a disjunction of conjunctions of
worms or negated worms such that non-empty worms that are not negated have a
first modality $\alpha$ and non-empty worms that are negated have a first
modality strictly less than $\alpha$. Moreover, we have that
${\mods}({\sf Diamond}_{\alpha}(\phi))\subseteq \{ \alpha \}
\cup {\mods}(\phi)$.
\end{lemma}

\begin{proof}
All modal reasoning concerns $\gllam$. Any word $A_i$ in $\phi(A_1, \ldots, A_n)$ is equivalent to
some ${B_i}\wedge {C_i}$ where ${B_i}\in \worms_{\alpha}$ and such that
the first element of ${C_i}$ is less than $\alpha$. Thus,
$\phi(A_1, \ldots , A_n)$ is equivalent to some other Boolean
combination $\psi({B_1}, \ldots , {B_n}, {C_1}, \ldots , {C_n})$ of
the worms ${B_1}, \ldots , {B_n}, {C_1}, \ldots , {C_n}$.

We write $\psi$ in disjunctive normal form. In the remainder of this proof we shall not be too precise in writing indices and subindices as the context should make clear what is meant. As $\langle \alpha
\rangle  \bigvee_j \chi_j \leftrightarrow \bigvee_i \langle \alpha
\rangle \chi_j$, it suffices to prove the lemma for formulas of the
form $\bigwedge_i \pm D_i$ where each $D_i \in \{ {B_1}, \ldots ,
{B_n}, {C_1}, \ldots , {C_n} \}$. By Lemma \ref{theorem:basicLemma}
we see that
\[
\langle \alpha \rangle \bigwedge_i \pm D_i \ \leftrightarrow \
\bigwedge_j \pm {C_i} \wedge \langle \alpha \rangle \bigwedge_k
\pm {B_i}.
\]
As worms are closed under taking conjunctions, we can write
$\bigwedge_k \pm {B_i}$ of the form $B \wedge \bigwedge_l \neg C_l$
where each of $B, C_l \in \worms_{\alpha}$.

Now we can apply Lemma \ref{theorem:WormsClosedUnderDiamonds} to
obtain $\langle \alpha \rangle \bigwedge_k \pm {B_k} \leftrightarrow
\langle \alpha \rangle B$, and
\[
\langle \alpha \rangle \bigwedge_i \pm D_i \ \leftrightarrow \
\bigwedge_j \pm {C_j} \wedge \langle \alpha \rangle B.
\]
All the positive worms in $\bigwedge_i \pm {C_i}$ can be moved
as conjunctions under the $\langle \alpha \rangle$ modality of $\langle \alpha \rangle B$ again to form a single worm as the conjunctions of all
those worms are equivalent to a single one.
\end{proof}

\begin{corollary}\label{theorem:eachClosedFormulaEquivalentToWorms}
Each closed formula $\phi$ is equivalent in $\gllam$ to a Boolean
combination ${\sf BCW}(\phi)$ of worms such that ${\mods}({\sf
BCW}(\phi)) \subseteq {\mods}(\phi)$.
\end{corollary}

\begin{proof}
By induction on the complexity of $\psi$. The only interesting case
is $\langle \alpha \rangle$ which is taken care of by Lemma
\ref{theorem:closedFormulasClosedUnderDiamonds}. Note that in
principle ${\sf BCW}(\phi)$ need not be unique as, for example, one
could consider various equivalent disjunctive normal forms along the
way of constructing ${\sf BCW}(\phi)$.
\end{proof}

\begin{corollary}\label{theorem:eachClosedFormulaEffectivelyEquivalentToWorms}
For each closed formula $\psi$ of $\gllam$ we can
$\Lambda$-effectively compute an $\gllam$-equivalent formula $\chi$
which is a Boolean combination of worms such that
${\mods}(\chi) \subseteq {\mods}(\phi)$.
\end{corollary}

\begin{proof}
By inspection of the proofs of Lemma
\ref{theorem:closedFormulasClosedUnderDiamonds} and Lemma
\ref{theorem:WormsClosedUnderDiamonds} we can retrieve a
$\Lambda$-effective recipe. We use that we already know that we can
$\Lambda$-effectively compare two worms and compute their
conjunction.
\end{proof}

\begin{corollary}\label{theorem:zeroDiamondYieldWorms}
For each consistent closed formula $\phi$ there is a worm $A$ with
${\mods}(A) = {\mods}(\phi)\cup \{ 0\}$ so that $\gllam
\vdash \langle 0\rangle \phi \leftrightarrow A.$

Moreover, $\gllam^0 \vdash \langle \max(A)\rangle^{l(A)}
\top \ \to \ A.$
\end{corollary}

\begin{proof}

Write
$\phi$ in disjunctive normal form where the
atoms are worms. As $\la 0\ra$ distributes over our disjunction, to each disjunct we apply Lemma
\ref{theorem:WormsClosedUnderDiamonds}. As $\varphi$ was consistent, so is each of the disjuncts whence  each disjunct is equivalent $\langle 0\rangle A_i$ for some
worm $A_i$. Thus, we end up with a disjunction of worms that start with a
$\langle 0\rangle$ modality.  Corollary \ref{theorem:ComparingWorms}
tells us that there is a `minimal' disjunct and thus we see that
such a disjunction can actually be replaced by a single disjunct.

By an easy proof similar to that of Lemma
\ref{theorem:RemoveMultipleMinimals}, we further see that
\[
\gllam^0 \vdash \langle \max(A)\rangle^{l(A)} \top \ \to \
A,
\]
from which our second claim immediately follows.
\end{proof}

Corollary \ref{theorem:zeroDiamondYieldWorms} has an important
consequence for the model theory of $\gllam^0$. This result is used in
\cite{FernandezJoosten:2012:ModelsOfGLP} to give a completeness
proof for certain models of the closed fragment. Namely, if we have
a Kripke frame $\mathfrak F$ such that $\mathfrak F\models
\gllam^0$ and we wish to check that $\gllam^0$ is
moreover {\em complete} for $\mathfrak F$, it suffices to check that
$\mathfrak F$ satisfies enough worms:

\begin{corollary}\label{theorem:CompletenessFollowsFromSoundness}
Suppose $\mathfrak F=\langle W,\langle
R_\xi\rangle_{\xi<\Lambda}\rangle$ is any Kripke frame such that
$\mathfrak F\models\glp^0_\Lambda$ and, for all $\lambda<\Lambda$
and $n<\omega$, there is $w\in W$ such that $\mathfrak
F,w \models \langle \lambda\rangle^n\top$.

Then, for every consistent closed formula $\phi$ there is $w\in W$
such that $\mathfrak F,w\models\phi$.

If $\Lambda$ is a limit ordinal, it suffices to consider $n=1$.
\end{corollary}

\proof Suppose that $\mathfrak F,w \models \langle
\lambda\rangle^n\top$ for all $n<\omega$ and $\lambda<\Lambda$ and
$\phi$ is consistent.

Then we have in particular that for some $w\in W$, $\mathfrak
F,v\models \langle \max\phi\rangle^{l(\phi)}\top$, so that by
Corollary \ref{theorem:zeroDiamondYieldWorms} we also have
$\mathfrak F,v\models \langle 0\rangle\phi$. But then
we have $w$ with $v\mathrel R_0 w$ and $\mathfrak
F,w\models \phi$, i.e., $\phi$ is satisfied on $\mathfrak F$,
as claimed.

If $\Lambda$ is a limit ordinal we observe that
\[
\glp^0_\Lambda \vdash \langle\max\phi+1\rangle
\top \rightarrow \langle \max\phi\rangle^{l(\phi)}\top,
\]
so we may
choose $v$ satisfying $\langle\max\phi+1\rangle \top$ instead.
\endproof

Note that this corollary is here stated for Kripke semantics but actually holds true for any reasonable notion of $\gllam$ semantics.


\section{Alternative axiomatizations}

In \cite{BeklemishevJoostenVervoort:2005:FinitaryTreatmentGLP} it
was observed that one could simultaneously restrict L\"ob's axiom
and the monotonicity axiom $\la \alpha \ra \phi \to \la \beta \ra
\phi$ for $\alpha \geq \beta$ to worms and still obtain a full
axiomatization of $\glp_\omega^0$. In this section we shall prove
that we can also simultaneously restrict the axiom of negative
introspection $\la \alpha \ra \phi \to [\beta ]\la \alpha \ra \phi$
with $\alpha < \beta$ to worms and still obtain a full
axiomatization of $\glp_\omega^0$. In order to prove this, we need
to recall the decision procedure as exposed in
\cite{Beklemishev:2005:VeblenInGLP}.

\subsection{A decision procedure}

\begin{theorem}\label{theorem:DecisionProcedure}
There is a $\Lambda$-effective decision procedure for
$\gllam^0\vdash \phi$.
\end{theorem}

\begin{proof}
We shall first outline a decision procedure and then see that this
is indeed effective. By Corollary
\ref{theorem:eachClosedFormulaEquivalentToWorms} we know that each
closed formula $\phi$ is equivalent in $\gllam$ to a Boolean
combination of worms. We can write this Boolean combination in
conjunctive normal form and as worms are closed under conjunctions,
each conjunct can be written of the form $A_i \to \bigvee_j B_{ij}$
with each $A_i$ and $B_{ij}$ in WNF. Let us call this the \emph{worm
normal form} and we write ${\sf WNF}(\varphi)$.

The decision procedure is represented by the following scheme:

\[
\begin{array}{llll}
\gllam^0 \vdash \phi & \Leftrightarrow & \gllam^0 \vdash {\sf WNF}(\phi)& \\
 & \Leftrightarrow & \gllam^0 \vdash \bigwedge_i (A_i \to \bigvee_jB_{ij}) & \\
 & \Leftrightarrow & \forall i\ \gllam^0 \vdash A_i \to \bigvee_j B_{ij} & \mbox{by Lemma \ref{theorem:DisjunctionPropertyWorms}}\\
 & \Leftrightarrow & \forall i\, \exists j\  \gllam^0 \vdash A_i \to B_{ij} & \\
 & \Leftrightarrow & \forall i\, \exists j\  \gllam^0 \vdash A_i \leftrightarrow A_i\wedge  B_{ij} & \mbox{by Lemma \ref{theorem:WNFsAreUnique}}\\
  & & &\mbox{and Corollary \ref{theorem:WormsClosedUnderConjunction}}\\
& \Leftrightarrow & {\mods}(\phi) \subseteq \Lambda \ \mbox{ and } \ & \\
& &\forall i\, \exists j\ {\sf NF}(A_i) = {\sf NF}(A_i \wedge
B_{ij})
\end{array}
\]
The ${\mods}(\phi) \subseteq \Lambda$ in the last line we have
in virtue of our conservation result as stated in \ref{conserv}. In
order to see that the above equivalences yield a $\Lambda$-effective
decision procedure, there are three major things that we need to
check.
\begin{enumerate}
\item\label{item:eachClosedFormulaEffectivelyEquivalentToWorms}
${\sf WNF}(\phi)$ can be $\Lambda$-effectively computed from a
closed formula $\phi$;
\item\label{item:WNFisEffectivelyComputable}
${\sf NF}(A)$  can be $\Lambda$-effectively computed from a worm
$A$;
\item\label{item:effectiveWormConjunction}
The worm corresponding to $A\wedge B$ can be $\Lambda$-effectively
computed from $A$ and $B$.
\end{enumerate}
But, Item \ref{item:effectiveWormConjunction} is just Corollary
\ref{theorem:effectiveWormConjunction}, Item
\ref{item:WNFisEffectivelyComputable} is just Corollary
\ref{theorem:WNFisEffectivelyComputable}, and Item
\ref{item:eachClosedFormulaEffectivelyEquivalentToWorms} follows
directly from Corollary
\ref{theorem:eachClosedFormulaEffectivelyEquivalentToWorms} and
Corollary \ref{theorem:effectiveWormConjunction}.
\end{proof}

In practice we will always only be interested in notation systems
that are easy, say primitive recursive, for which the following
corollary is relevant.

\begin{corollary}
For each effective ordinal $\Lambda$, there is an effective decision
procedure for $\gllam^0\vdash \phi$.
\end{corollary}

In virtue of Theorem \ref{conserv} we knew already that
$\glp^0_\Lambda$ has a very easy reduction to $\glp^0_\omega$ where
the latter is know tho be {\sc PSpace} complete.

\begin{corollary}
If the ordering on $\Lambda$ is decidable in poly-time, then the
computational complexity of $\gllam^0$ is {\sc PSpace}
complete.
\end{corollary}

\begin{proof}
Theorem \ref{conserv}, provides a poly-time reduction from
$\gllam^0$ to $\glp_{\omega}^0$. Although the closed
fragment for \gl is decidable in {\sc PTime}
(\cite{ChagrovRybakov:2003:HowManyVariablesForPSPACE}), Pakhomov has
shown (\cite{Pakhomov:2011:ComplexityResultsInGLP}) that the closed
fragment of $\glp_{\omega}$ is {\sc PSpace} complete.
\end{proof}

\subsection{Restricting to worms}

We are now ready to prove the main theorem of this section. By
$\mathsf{w{-}GLP_\Lambda^0}$ we denote the logic that is as
$\glp^0_\Lambda$ but the axioms
\[
\begin{array}{ll}
[\alpha] ([\alpha] A \to A) \to [\alpha] A & \\
\la \alpha \ra A \to \la \beta \ra A & \alpha \geq \beta\\
\la \alpha \ra A \to [\beta ]\la \alpha \ra A & \alpha < \beta\\
\end{array}
\]
restricted to worms $A$.

\begin{theorem}\label{theorem:ConservativeFragments}
The logics $\mathsf{w{-}GLP_\Lambda^0}$ and $\gllam^0$ prove
the same set of theorems.
\end{theorem}

\begin{proof}
We will first prove
\[
\begin{array}{ll}
\langle \alpha \rangle \phi \to \langle \beta \rangle \phi& \mbox{ for $\alpha \geq \beta$ and }\\
\langle \alpha \rangle \phi \to [\beta] \langle \alpha \rangle \phi
& \mbox{for $\alpha < \beta$}
\end{array}
\]
for $\phi$ any closed formula within $\mathsf{w{-}GLP_\Lambda^0}$.
We  write $\phi$ in disjunctive normal form as $\bigvee_i  (A_i
\wedge \bigwedge_j \neg B_{ij} \wedge \bigwedge_k \neg C_{ik})$
where each $B_{ij} \in \worms_{\alpha}$ and each $C_{ik}$ starts with a
modality smaller than $\alpha$.

When $\langle \alpha \rangle \phi \leftrightarrow \bot$ there is
nothing to prove, so we may assume that $\nvdash A_i \to B_{ij}$ and
$\nvdash A_i \to C_{ik}$ and use Corollary
\ref{theorem:generalizedWormsClosedUnderDiamonds} to see that for
each $i$ we have that
\[
\begin{array}{lll}
\langle \alpha \rangle (A_i \wedge \bigwedge_j \neg B_{ij} \wedge \bigwedge_k \neg C_{ik})& \leftrightarrow & \langle \alpha \rangle (A_i \wedge \bigwedge_j \neg B_{ij} \wedge \bigwedge_k \neg C_{ik})\\
& \leftrightarrow &  \bigwedge_k \neg C_{ik} \wedge \langle \alpha \rangle (A_i \wedge \bigwedge_j \neg B_{ij} )\\
& \leftrightarrow &  \bigwedge_k \neg C_{ik} \wedge \langle \alpha \rangle A_i \\
& \leftrightarrow &   \langle \alpha \rangle (A_i \wedge \bigwedge_k \neg C_{ik}).\\
\end{array}
\]
Let us first see that $\langle \alpha \rangle \phi \to \langle \beta
\rangle \phi \mbox{ for $\alpha \geq \beta$}$. We observe that $\worms_{\alpha} \subset \worms_{\beta}$. We shall write $\bigwedge_k
\neg C_{ik}$ as $\bigwedge_{k'} \neg C_{ik'} \wedge \bigwedge_l \neg
D_{il}$ where the first modality in each $C_{ik'}$ is strictly below
$\beta$ and the first modality in each $D_{il}$ is between $\beta$
and strictly below $\alpha$.
\[
\begin{array}{llll}
\langle \alpha \rangle \phi & \to & \langle \alpha \rangle \bigvee_i  (A_i \wedge \bigwedge_j \neg B_{ij} \wedge \bigwedge_k \neg C_{ik})\\
 & \to &  \bigvee_i  \langle \alpha \rangle(A_i \wedge \bigwedge_j \neg B_{ij} \wedge \bigwedge_k \neg C_{ik})\\
 & \to &  \bigvee_i  (\bigwedge_k \neg C_{ik} \wedge \langle \alpha \rangle (A_i \wedge \bigwedge_j \neg B_{ij} ))\\
 & \to  &  \bigvee_i (\bigwedge_k \neg C_{ik} \wedge \langle \alpha \rangle A_i )\\
 & \to &   \bigvee_i( \bigwedge_k \neg C_{ik} \wedge \langle \beta \rangle A_i )\\
 & \to & \bigvee_i (\bigwedge_{k'} \neg C_{ik'} \wedge \bigwedge_l \neg D_{il}\wedge \langle \beta \rangle A_i) \\
 & \to  & \bigvee_i (\bigwedge_{k'} \neg C_{ik'} \wedge \langle \beta \rangle A_i ) \\
 & \to & \bigvee_i \langle \beta \rangle (A_i \wedge \bigwedge_{k'} \neg C_{ik'}))& \mbox{As $\nvdash A_i \to D_{il}$}\\
 & \to & \bigvee_i  \langle \beta \rangle (A_i \wedge \bigwedge_j \neg B_{ij} \wedge \bigwedge_{k'} \neg C_{ik'} \wedge\bigwedge_l \neg D_{il})\\
 & \to &   \langle \beta \rangle \bigvee_i(A_i \wedge \bigwedge_j \neg B_{ij} \wedge \bigwedge_{k'} \neg C_{ik'} \wedge\bigwedge_l \neg D_{il})\\

 & \to &  \langle \beta \rangle \phi.
\end{array}
\]

For the proof of $\langle \alpha \rangle \phi \to [\beta] \langle
\alpha \rangle \phi  \mbox{ for $\alpha < \beta$}$ it clearly
suffices to show for each $i$ that
\[
\langle \alpha \rangle(A_i \wedge \bigwedge_j \neg B_{ij} \wedge
\bigwedge_k \neg C_{ik}) \to [\beta] \langle \alpha \rangle
\bigvee_i (A_i \wedge \bigwedge_j \neg B_{ij} \wedge \bigwedge_k
\neg C_{ik}).
\]
To establish this we observe that $\vdash \neg C_{ik} \to [\beta]
\neg C_{ik}$ and use large part of our reasoning before:
\[
\begin{array}{llll}
\langle \alpha \rangle(A_i \wedge \bigwedge_j \neg B_{ij} \wedge \bigwedge_k \neg C_{ik}) & \to & \bigwedge_{k} \neg C_{ik} \wedge \langle \alpha \rangle A_i\\
 & \to & \bigwedge_{k} \neg C_{ik} \wedge [\beta] \langle \alpha \rangle A_i\\
  & \to & \bigwedge_{k} [\beta] \neg C_{ik} \wedge [\beta] \langle \alpha \rangle A_i\\
 & \to & [\beta](\bigwedge_{k}  \neg C_{ik}) \wedge [\beta] \langle \alpha \rangle A_i\\
 & \to & [\beta](\bigwedge_{k}  \neg C_{ik}) \wedge [\beta] \langle \alpha \rangle (A_i\wedge \bigwedge_j \neg B_{ij})\\
 & \to &  [\beta]  (\bigwedge_{k}  \neg C_{ik}\wedge \langle \alpha \rangle (A_i\wedge \bigwedge_j \neg B_{ij}))\\
 & \to &  [\beta]  \langle \alpha \rangle (A_i\wedge \bigwedge_j \neg B_{ij} \wedge \bigwedge_{k}  \neg C_{ik})\\
 & \to &  [\beta]  \langle \alpha \rangle \bigvee_i (A_i\wedge \bigwedge_j \neg B_{ij} \wedge \bigwedge_{k}  \neg C_{ik}).\\

\end{array}
\]

Giving an explicit proof for the full version of L\"ob's axiom from
the restricted ones seems to be rather involved thus we choose
another proof strategy.

We observe that the only (!) application of L\"ob's axiom in this
paper is in Lemma \ref{theorem:WormsClosedUnderDiamonds} where it is
actually restricted to worms. Thus, with the restricted version of
L\"ob's axiom we come to the same decision procedure and the same
set of unique WNFs whence the two logics
$\mathsf{w{-}GLP_\Lambda^0}$ and $\gllam^0$ prove the same set
of theorems.
\end{proof}

\section{Acknowledgements}
The first author was supported by the Russian Foundation for Basic Research (RFBR), the Presidential council for support of leading scientific schools, and the Swiss--Russian
cooperation project STCP--CH--RU ``Computational proof theory''.

The second author was supported by the project ``Alternative interpretations of non-classical logics'' (HUM-5844) of the {\em Junta de Andaluc\'ia.}

The third author wishes to thank the participants of the Seminari Cuc in Barcelona for feedback, question, suggestions and discussions: Joan Bagaria, Felix Bou, Ramon Jansana and Enrique Casanovas.

\bibliographystyle{plain}
\bibliography{References}

\begin{thebibliography}{10}

\bibitem{BeklemishevGabelaia:2011:TopologicalCompletenessGLP}
L.~D. Beklemishev and D.~Gabelaia.
\newblock Topological completeness of the provability logic $\mathsf{GLP}$.
\newblock {\em ArXiv}, 1106.5693v1 [math.LO], 2011.
\newblock To appear in \emph{Annals of Pure and Applied Logic}.

\bibitem{BeklemishevJoostenVervoort:2005:FinitaryTreatmentGLP}
L.~D. Beklemishev, J.~J. Joosten, and M.~Vervoort.
\newblock A finitary treatment of the closed fragment of {J}aparidze's
  provability logic.
\newblock {\em Journal of Logic and Computation}, 15:447--463, 2005.

\bibitem{Beklemishev:2004}
L.D. Beklemishev.
\newblock Provability algebras and proof-theoretic ordinals, {I}.
\newblock {\em Annals of Pure and Applied Logic}, 128:103--124, 2004.

\bibitem{Beklemishev:2005:VeblenInGLP}
L.D. Beklemishev.
\newblock Veblen hierarchy in the context of provability algebras.
\newblock In P.~H\'ajek, L.~Vald\'es-Villanueva, and D.~Westerst{\aa}hl,
  editors, {\em Logic, Methodology and Philosophy of Science, Proceedings of
  the Twelfth International Congress}, pages 65--78. Kings College
  Publications, 2005.

\bibitem{Beklemishev:2006}
L.D. Beklemishev.
\newblock The {Worm} principle.
\newblock In Z.~Chatzidakis, P.~Koepke, and W.~Pohlers, editors, {\em Logic
  Colloquium 2002, Lecture Notes in Logic 27}, pages 75--95. ASL Publications,
  2006.

\bibitem{Beklemishev:2010}
L.D. Beklemishev.
\newblock Kripke semantics for provability logic \glp.
\newblock {\em Annals of Pure and Applied Logic}, 161(6):737--744, 2010.

\bibitem{Beklemishev:2010:OnCraigInterpolation}
L.D. Beklemishev.
\newblock On the {Craig} interpolation and the fixed point properties of {GLP}.
\newblock In S.~Feferman et~al., editor, {\em Proofs, Categories and
  Computations. Essays in honor of G. Mints}, Tributes, pages 49--60. College
  Publications, London, 2010.
\newblock Preprint: Logic Group Preprint Series 262, University of Utrecht,
  Dec. 2007.

\bibitem{Beklemishev:2011:SimplifiedArithmeticalCompleteness}
L.D. Beklemishev.
\newblock A simplified proof of the arithmetical completeness theorem for the
  provability logic \glp.
\newblock {\em Trudy Matematicheskogo Instituta imeni V.A. Steklova},
  274(3):32--40, 2011.
\newblock English translation: \emph{Proceedings of the Steklov Institute of
  Mathematics}, 274(3):25--33, 2011.

\bibitem{Boolos:1993:LogicOfProvability}
G.~S. Boolos.
\newblock {\em The {L}ogic of {P}rovability}.
\newblock Cambridge University Press, Cambridge, 1993.

\bibitem{ChagrovRybakov:2003:HowManyVariablesForPSPACE}
{Chagrov, A. V. and Rybakov, M. N.}
\newblock How many variables does one need to prove \textsc{PSpace}-hardness of
  modal logics.
\newblock In {\em Advances in Modal Logic}, volume~4, pages 71--82, 2003.

\bibitem{FefSpe62}
S.~Feferman and C.~Spector.
\newblock Incompleteness along paths in progressions of theories.
\newblock {\em The Journal of Symbolic Logic}, 27:383--390, 1962.

\bibitem{Fernandez:2012:TopologicalCompleteness}
D.~Fern\'andez-Duque.
\newblock The polytopologies of transfinite provability logic.
\newblock {\em ArXiv}, 1207.6595 [math.LO], 2012.

\bibitem{FernandezJoosten:2012:ModelsOfGLP}
D.~Fern\'andez-Duque and J.~J. Joosten.
\newblock Models of transfinite provability logics.
\newblock {\em Journal of Symbolic Logic}, 2012.
\newblock Accepted for publication.

\bibitem{FernandezJoosten:2012:Hyperations}
{Fern\'andez-Duque, D. and Joosten, J. J.}
\newblock {Hyperations, {V}eblen progressions and transfinite iteration of
  ordinal functions}.
\newblock Submitted, May 2012.

\bibitem{FernandezJoosten:2012:KripkeSemanticsGLP0}
{Fern\'andez-Duque, D. and Joosten, J. J.}
\newblock Kripke models of transfinite provability logic.
\newblock In {\em Advances in Modal Logic}, volume~9, pages 185--199. College
  Publications, 2012.

\bibitem{FernandezJoosten:2012:TuringProgressions}
{Fern\'andez-Duque, D. and Joosten, J. J.}
\newblock Turing progressions and their well-orders.
\newblock In {\em How the world computes}, Lecture Notes in Computer Science,
  pages 212--221. Springer, 2012.

\bibitem{FernandezJoosten:2012:WellOrders2}
{Fern\'andez-Duque, D. and Joosten, J. J.}
\newblock {Well-orders in the transfinite {J}aparidze algebra {II}}.
\newblock forthcoming, 2012.

\bibitem{Icard:2008:MastersThesis}
T.~F. Icard~III.
\newblock Models of the polymodal provability logic.
\newblock Master's thesis, Institute for Logic Language and Information, 2008.

\bibitem{Icard:2009:TopologyGLP}
T.~F. Icard~III.
\newblock A topological study of the closed fragment of $\mathsf{GLP}$.
\newblock {\em Journal of Logic and Computation}, 21:683--696, 2011.

\bibitem{Ignatiev:1993:StrongProvabilityPredicates}
K.~N. Ignatiev.
\newblock On strong provability predicates and the associated modal logics.
\newblock {\em The Journal of Symbolic Logic}, 58:249--290, 1993.

\bibitem{Japaridze:1986}
G.K. Japaridze.
\newblock {\em The modal logical means of investigation of provability}.
\newblock PhD thesis, Moscow State University, 1986.
\newblock In {R}ussian.

\bibitem{Joosten:2004:InterpretabilityFormalized}
J.~J. Joosten.
\newblock {\em Intepretability Formalized}.
\newblock PhD thesis, Utrecht University, 2004.

\bibitem{Pakhomov:2011:ComplexityResultsInGLP}
F.~Pakhomov.
\newblock On the complexity of the closed fragment of {J}aparidze's provability
  logic.
\newblock In T.~Bolander, T.~Bra\"{u}ner, S.~Ghilardi, and L.~Moss, editors,
  {\em 9-th Advances in Modal Logic, AiML 2012, Short Presentations}, pages
  56--59, 2012.

\bibitem{Shamkanov:2011:InterpolationProperties}
D.S. Shamkanov.
\newblock Interpolation properties of provability logics {GL} and {GLP}.
\newblock {\em Trudy Matematicheskogo Instituta imeni V.A. Steklova},
  274(3):329--342, 2011.
\newblock English translation: \emph{Proceedings of the Steklov Institute of
  Mathematics}, 274(3):303--316, 2011.

\end{thebibliography}

\end{document}